\documentclass[a4paper,12pt]{amsart}

\usepackage{fullpage}
\usepackage[driverfallback=dvipdfm]{hyperref}
\usepackage[all]{xy}
\usepackage{tikz}
\usepackage{amsaddr}
\usepackage{amsthm, amssymb}
\usepackage{color}
\usepackage{enumerate}

\usepackage{mathtools}
\usepackage{thmtools,thm-restate}

\newtheorem{thm}{Theorem}[section]
\newtheorem*{thm*}{Theorem}
\newtheorem{cor}[thm]{Corollary}
\newtheorem{lem}[thm]{Lemma}
\newtheorem{prop}[thm]{Proposition}

\theoremstyle{definition}
\newtheorem{defn}[thm]{Definition}

\newtheorem{nt}[thm]{Notation}

\newtheorem{rem}[thm]{Remark}

\newtheorem*{rem*}{Remark}

\title[Minset of an irreducible automorphism of a free group is co-compact]{The mimimally displaced set of
  an irreducible automorphism of $F_N$ is co-compact}
\author{Stefano Francaviglia}
\address{Dipartimento di Matematica of the University of
Bologna}
\email{stefano.francaviglia@unibo.it}
\author{Armando Martino}
\address{Mathematical Sciences, University of Southampton }
\email{A.Martino@soton.ac.uk}

\author{Dionysios Syrigos}
\address{Mathematical Sciences, University of Southampton }
\email{d.Syrigos@soton.ac.uk}

\begin{document}


\begin{abstract}
We study the minimally displaced set of irreducible automorphisms of a free group. Our main result is the co-compactness of the minimally displaced set of an irreducible automorphism with exponential growth $\phi$, under the action of the centraliser $C(\phi)$. As a corollary, we get that the same holds for the action of $ <\phi>$ on $Min(\phi)$. Finally, we prove that the minimally displaced set of an irreducible automorphism of growth rate one is consisted of a single point.
\end{abstract}
\maketitle
\tableofcontents

\section{Introduction}

In this paper, we study the minimally displaced set of an irreducible automorphism of a free group $F_N$. Our goal is to produce an essentially elementary proof that an irreducible automorphism acts co-compactly on its minimally displaced set (equivalently, the train track points) in Culler-Vogtmann space. We note that stronger results were produced by  \cite{Handel-Mosher}, who showed that the action on the {\em axis bundle} is co-compact, from which one easily deduces the result. However, their result only holds for non-geometric iwip automorphisms, and we believe our argument to be significantly simpler.

Irreducible automorphisms of $F_N$ play a central role in the study of the of outer automorphism group $Out(F_N)$, as they can be studied using the powerful train track machinery, which has been introduced by Bestvina and Handel in \cite{BH-TrainTracks}. Also, they are generic elements of $Out(F_N)$ in the sense of random walks (see \cite{Rivin}).

Culler-Vogtmann Outer space $CV_N$ is considered a classical method to understand automorphisms of free groups, by studying the action of $Out(F_N)$ on $CV_N$. More recently, there is an extensive study of the so-called Lipschitz metric on $CV_N$ which seems that it has interesting applications (for example, see \cite{Kfir-Bestvina}, \cite{BestvinaBers}, \cite{FM12}, \cite{QingRafi}).

In particular, given an automorphism $\phi$, we can define the displacement of $\phi$ in $CV_N$, as $\lambda_{\phi} = \inf\{  \Lambda(X,(X)\phi) : X \in CV_N\}$ (where we denote by $\Lambda$ the asymmetric Lipschitz metric). The set of minimally displaced points (or simply Min-Set) $Min(\phi)$ has interesting properties. For instance, the first two authors studied the Min-Set in \cite{FM18II} (in the context of general deformation spaces), where they proved that it is connected. As an application they gave solutions to some decision problems for irreducible automorphisms.


For an automorphism $\phi$ the centraliser, $C(\phi)$, preserves the Min-Set, $Min(\phi)$. Our main result is the following:

\begin{restatable*}{thm}{centralcocompact}
Let $\phi$ be an irreducible automorphism of $F_N$ with $\lambda_{\phi} > 1$. The quotient space $Min(\phi) / C(\phi)$ is compact.
\end{restatable*}

As we have already mentioned, Handel and Mosher prove the previous theorem (under some extra hypotheses) in \cite{Handel-Mosher}. It is worth mentioning that some of the main ingredients of their proof is generalised by Bestina, Guirardel and Horbez in the context of deformation spaces of free products (see \cite{BGH}).

We then collect known results together, showing that the centraliser of an irreducible automorphism with $\lambda_{\phi} > 1$ is virtually cyclic, and deduce the following. 

\begin{restatable*}{thm}{autococompact}
	\label{autococom}
	Let $\phi$ be an irreducible automorphism of $F_N$ with $\lambda_{\phi} > 1$. Then $Min(\phi) / <\phi>$ is compact.
\end{restatable*}

\begin{rem*}
	Note that centralisers for iwip automorphisms are well known to be virtually cyclic, but the more general statement is also true since irreducible automorphisms of infinite order which are not iwip are geometric, and their centralisers are also geometric, hence virtually cyclic. We collect these observations in more detail in the proof of Theorem~\ref{autococom}.
\end{rem*}

In order to have a complete picture for irreducible automorphisms of a free group, we study irreducible automorphisms of growth rate one (all of these have finite order). In this case, we have the following:
\begin{restatable*}{thm}{lambdaonefix}
	Let $\phi$ be an irreducible automorphism of $F_N$ with $\lambda_{\phi} = 1$. There is a single point $T \in CV_N$ so that $Min(\phi) = Fix(\phi) = \{T\}$.
\end{restatable*}

As an application, we get the following corollary:
\begin{restatable*}{cor}{centrallambdaonefinite}
	Let $\phi$ be an irreducible automorphism of $F_N$ with $\lambda_{\phi} = 1$. There is some $T \in CV_N$ so that  $C(\phi)$ fixes $T$. In particular, $C(\phi)$ is finite.
\end{restatable*}

\section{Preliminaries}
\subsection{Culler -Vogtmann space}
For the rest of the paper, we will denote by $F_N$ the free group on $N$ generators, for some $N \geq 2$. Firstly, we will describe the construction of the Culler-Vogtmann space which is denoted by $CV_N$ and it is a space on which $Out(F_N)$ acts nicely.

Let's fix a free basis $x_1,\ldots,x_N$ of $F_N$. We denote by $R_N$ the rose with $n$-petals where we identify each $x_i$ with a single petal of $R_N$ (the petals are still denoted by $x_1,\ldots,x_N$).

\begin{defn}
A marked metric graph of rank $n$, is a triple $(T, h, \ell_T)$ so that :
\begin{itemize}
	\item $T$ is a graph (without valence one or two vertices) with fundamental group isomorphic to $F_N$. 
	\item $h: R_N \rightarrow T$ is a homotopy equivalence, which is called the marking
	\item $\ell_T : E(T) \rightarrow (0,1)$ is a metric on $T$, which assigns a positive number on each edge of $T$, with the property $\sum_e \ell_T(e) = 1$.
\end{itemize} 
\end{defn}
We are now in position to define Culler Vogtmann space:
\begin{defn}
$CV_N$ is the space of equivalence classes, under $\sim$, of marked metric graphs of rank $n$. The equivalence relation $\sim$ is given by $(T, h, \ell_T) \sim (S,h',\ell_S)$ if and only if there is an isometry $g:T \rightarrow S$ so that $gh$ is homotopic to $h'$.
\end{defn}
In order to simplify the notation, whenever $\ell_T,h$ are clear by the context, we can simply write $T$ instead of the triple $(T,h,\ell_T)$.

\begin{rem*}
	Equivalently, one may take universal covers to get a different formulation of Culler Vogtmann space; as the space of free, minimal, simplicial $F_N$ trees of volume 1, up to equivariant isometry. 
\end{rem*}

\textbf{Action of $Out(F_N)$}: Let $\phi \in Out(F_N)$ and $(T,h,\ell_T)$ be a marked metric graph. We define $((T,h, \ell_T))\phi = (T, h' , \ell_T)$ where $h ' = h \phi$ (here we still denote by $\phi$, the natural representative of $\phi$ as a homotopy equivalence from $R_N$ to $R_N$).

\textbf{Simplicial Structure.} If we fix a pair $(T,h)$ of a topological marked graph (without metric) and we consider all the different possible metrics on $E(T)$ (i.e. assignments of a positive number on each edge so that the volume of $T$ is one), we obtain an (open) simplex in $CV_N$ with dimension $|E(T)| -1$. However, if we allow some edges to have length $0$, then it is not always true that the new graph will have rank $n$, so the resulting simplex is not necessarily in $CV_N$ (note that some faces could be in $CV_N$). Therefore, $CV_N$ is not a simplicial complex, but it can be described as a union of open simplices. Note that it follows immediately by the definitions, that $Out(F_N)$ acts simplicially on $CV_N$. In the following theorem, we list some more properties which are proved in \cite{cv}.

\begin{thm*}[See {\cite{cv}}]
	
\noindent	
\begin{enumerate}[{i)}]
	\item $CV_N$ is contractible.
	\item The maximum dimension of a simplex is $3N-4$. In other words, the maximum number of edges in a marked metric graph with rank $N$, is bounded above by $3N-3$.
	\item $Out(F_N)$ acts on $CV_N$ by finite point stabilisers and the quotient space is consisted of finitely many open simplices
\end{enumerate}
\end{thm*}

\textbf{Translation length.} Given a marked metric graph $T$, we would like to define the translation length of the conjugacy class of a non-trivial group element $a$ with respect to $T$. It is natural to define the length of $a$, as the sum of lengths of the edges that $a$ crosses when it is realised as a reduced loop in $T$. We define the translation length of the conjugacy class of $a \in F_N$, as  $\ell_T ([a]) = \inf \{len_T(a'): [a'] =[a]   \}$. It is easy to see that any group element is freely homotopic to an embedded loop which realises the minimum.

\begin{rem*}
	This translation length is the same as the translation length - in the sense of the minimum distance moved by the group element - of the corresponding action of the group element on the universal cover, a tree.
\end{rem*}

\textbf{Thick part.} We can now define the thick part of $CV_N$, which will be essential in our arguments:
\begin{defn}
Let $\epsilon > 0$. We define the thick part of $CV_N$, and we denote it by $CV_N (\epsilon)$, the subspace of $CV_N$ which is consisted of all the points $T \in CV_N$ with the property that every non-trivial conjugacy class $\alpha$ of $F_N$ has translation length at least $\epsilon$, i.e. $\ell_T(\alpha) \geq \epsilon$.
\end{defn}

\textbf{Centre of a simplex.} We will need the notion of a special point of a simplex, the centre, so we give the following definition:
\begin{defn}
	Let $\Delta$ be a simplex of $CV_N$. The point of $\Delta$ where all the edges have the same length is called the centre of $\Delta$ and we will denote it by $X_{\Delta}$.
\end{defn}

\begin{rem}\label{Centres}
	By the definition of the action of $Out(F_N)$ on $CV_N$, for any automorphism $\psi$, $X_{\psi(\Delta)} = \psi(X_{\Delta}) $, i.e. the centre of $\Delta$ is sent to the centre of $\psi(\Delta)$.
\end{rem}

\subsection{Stretching factor and Automorphisms}
We define a natural notion of distance on $CV_N$, which has been studied in \cite{FM11}.

\begin{defn}
Let $T,S \in CV_N$. We define the (right) stretching factor as:
\begin{equation*}
\Lambda(T,S) = \sup \{\frac{\ell_S([a])}{\ell_T ([a])}: 1 \neq a \in F_N \}
\end{equation*}
\end{defn}

In the following Proposition, we state some main properties of the stretching factor:

\begin{prop}[See {\cite{FM11}}]
	
\noindent	
\begin{enumerate}[{i)}]
	\item For any two points $T,S \in CV_N$,  $\Lambda(T,S) \geq 1$, with equality if and only if $T = S$.
	\item For any $T,S,Q \in CV_N$, the (non-symmetric) multiplicative triangle inequality for $\Lambda$ holds. In other words, $\Lambda (T,S) \leq \Lambda(T,Q) \Lambda(Q,S)$.
	\item For any $\psi \in Out(F_N)$ and for any $T,S \in CV_N$, $\Lambda(T,S) = \Lambda( (T)\psi, (S)\psi)$
\end{enumerate}
\end{prop}

Note that $\Lambda$ is not symmetric, i.e. there are points $T,S$ of $CV_N $ so that $\Lambda(T,S) \neq \Lambda(S,T)$. In fact, it is not even quasi-symmetric, in general. However, it is proved in \cite{Kfir-Bestvina} that for any $\epsilon > 0$, its restriction on the $\epsilon$-thick part, $CV_N (\epsilon)$, induces a quasi-symmetric function. More specifically, we have the following:
\begin{prop}[{Theorem 24, \cite{Kfir-Bestvina}}]\label{quasi-symmetry}
For every $\epsilon > 0$ and for every $T,S \in CV_N (\epsilon)$, there is a uniform constant $K $ (depending only on $\epsilon$ and $N$) so that $\Lambda(T,S) \leq \Lambda(S,T)^K$.
\end{prop}

 One could consider the function $d_R(T,S) =  \ln \Lambda (T,S)$, which behaves as an asymmetric distance, or even the symmetrised version $d(T,S) = d_R(T,S) + d_R(S,T)$, which is a metric on $CV_N$. We choose to work with $\Lambda$, as it is related more naturally to the displacement of an automorphism which is defined in the next subsection. However, as we will work only with some fixed $\epsilon$-thick part of $CV_N$, it follows by the quasi-symmetry of $\Lambda$, that we could use exactly the same arguments using $d$, instead.\\
\textbf{Balls in Outer Space.}
The asymmetry implies that there are three different types of balls that we can define, which are different, in general.
\begin{enumerate}
	\item The symmetric closed ball with centre $T \in CV_N$ and radius $r> 0$:
	\begin{equation*}
	B (T,r) = \{X \in CV_N : \Lambda(X,T) \Lambda(T,X) \leq r \}
	\end{equation*}
	\item  The in-going ball with centre $T \in CV_N$ and radius $r>0$:
	\begin{equation*}
	B_{in} (T,r) = \{X \in CV_N : \Lambda(X,T) \leq r \}
	\end{equation*}
	\item The out-going ball with centre $T \in CV_N$ and radius $r>0$:
	\begin{equation*}
	B_{out} (T,r) = \{X \in CV_N : \Lambda(T,X) \leq r \}
	\end{equation*}
\end{enumerate}

\begin{prop}\label{Compactness of Balls}
Let $T \in CV_N$ and $r > 0$.
\begin{enumerate}[{i)}]
	\item The symmetric ball $B (T,r)$ is compact.
	\item Then in-going ball $B_{in}(T,r)$ is compact.
\end{enumerate}
\end{prop}
\begin{proof}

	\noindent
\begin{enumerate}[{i)}]
	\item The statement  is proved in \cite{FM11}, Theorem 4.12.
	\item Let $T,r$ as before. Firstly note that there is some $C$ (the injective radius of $T$) for which $\ell_T([g]) \geq C > 0$ for every non-trivial $g \in F_N $. 
	Therefore, for every non-trivial $g \in F_N$ and for every $X \in B_{in}(T,r)$, it holds that:
	\begin{equation*}
 \frac{\ell_T([g])}{\ell_X([g])} \leq 	 \Lambda(X,T)  \leq r  \Rightarrow \\ 
 \ell_X([g]) \geq \frac{\ell_T([g])}{r} \geq \frac{C}{r}
	\end{equation*}
As a consequence, it follows that $B_{in}(T,r) \subset CV_N (C/r)$. By quasi-symmetry of $\Lambda$ when it is restricted on the thick part (\ref{quasi-symmetry}), we have that there is some $M$ so that $$B_{in} (T,r) \subseteq B(T,r^{2M}) $$ The result now follows by i), as $B_{in}(T,r)$ is a closed subset of the compact set $B(T,r^{2M})$ and therefore compact.
\end{enumerate}
\end{proof}

Note that it's not difficult to find an out-going ball $B_{out}(T,r)$, for some $T \in CV_N$, $r > 0$, which is not compact. It is worth mentioning that $B_{out}$ is weakly convex, while $B_{in}$ is not, in general (as  is proved in \cite{QingRafi}).

In \cite{FM11}, it is proved that the supremum in the definition of stretching factor $\Lambda(T,S)$ is maximum, as there is a hyperbolic element that realises the supremum. Even more, such a hyperbolic element can be chosen by a finite list of candidates of $T$.
\begin{defn}
Let $T \in CV_N$. A hyperbolic element $a$ of $F_N$ is called a candidate with respect to $T$, if the realisation of $a$ as a loop in $T$ is either:
\begin{itemize}
	\item an embedded circle
	\item a figure eight
	\item a barbell
\end{itemize}
The set of candidates with respect to $T$ is called the set of candidates and it is denoted by $Cand(T)$.
\end{defn}

\begin{prop}[{Prop.3.15, [\cite{FM11}}]\label{Candidates}
For any $T,S \in CV_N$, there is a candidate $a \in Cand(T)$ so that:
\begin{equation*}
\Lambda(T,S) = \frac{\ell_S([a])}{\ell_T([a])}
\end{equation*} 
\end{prop}

\textbf{Displacement.}
Now let's fix an automorphism $\phi \in Out(F_N)$. The displacement of $\phi$ is $\lambda_{\phi} = \inf \{\Lambda(X, \phi(X)) : X \in CV_N\}$. 

\begin{defn}
Let $\phi$ be an outer automorphism of $F_N$. Then we define the Min-Set of $\phi$, as $Min(\phi) = \{T \in CV_N : \Lambda(T,\phi(T)) = \lambda_{\phi} \}$.
\end{defn}
Note that $Min(\phi)$ could be empty, in general. However, it is non-empty when $\phi$ is irreducible (see \cite{FM15} for more details).

\begin{defn}
Let $\phi$ be an automorphism of $F_N$. Then $\phi$ is called reducible, if there is a free product decomposition of $F_N = A_1 \ast \ldots \ast A_k \ast B $ (where $B$ could be empty) so that every $A_i$, $i=1,\ldots,k$ is proper free factor and the conjugacy classes of $A_i$'s are permuted by $\phi$. Otherwise, $\phi$ is called irreducible.
\end{defn}

We note that by [Th.8.19, \cite{FM15}], $Min(\phi)$ coincides with the set of train track points of $\phi$, when $\phi$ is irreducible. That is, the points of $CV_N$ that admit a (not necessarily simplicial) train track representative of $\phi$.

\subsection{Properties of the Thick Part}
In this subsection, we list some properties of the thick part that we will need in the following section.

\begin{prop}\label{Thickness1}
Let $\phi$ be an irreducible automorphism of $F_N$. Then there is a positive number $\epsilon_1 $ (depending only on $N$ and on $\phi$) so that $Min(\phi) \subset CV_N (\epsilon_1)$. One could take $\epsilon_1 = 1/ ((3N-3)\mu^{3N-2})$ for any $\mu > \lambda_{\phi}$.
\end{prop}
\begin{proof}
For a proof see Proposition 10 of \cite{BestvinaBers}.
\end{proof}

\begin{lem}\label{Thickness2}
Let $\Delta$ be a simplex of $CV_N$ and $T,S$ be two points of $\Delta$. If we further suppose that $T \in CV_N (\epsilon)$, then $\Lambda(T,S) \leq \frac{2}{\epsilon}$.
\end{lem}
\begin{proof}
The proof is an immediate corollary of the candidates theorem (\ref{Candidates}). Firstly, note that candidates depend only on the simplex and so for any $T,S \in \Delta$, $Cand(T) = Cand(S) = \mathcal{C}$. The result now follows by the remark that for every $g \in \mathcal{C}$ we have $\ell_S ([g]) \leq 2$, while for every $1 \neq g$, $\ell_T ([g]) \geq \epsilon$  as $T \in CV_N(\epsilon)$.
\end{proof}

\begin{rem}\label{Thickness of centre}
Let $\Delta$ be any simplex of $CV_N$, then there is some uniform positive number $\epsilon_2$ (depending only on $N$), so that the centre of $\Delta$ is $\epsilon_2$-thick. One could take, $\epsilon_2 = 1 / (3N-3)$. Therefore, it is easy to see $\epsilon_1 < \epsilon_2$ and so the centre of any simplex is $\epsilon_1$-thick.
\end{rem}
\begin{proof}
The number of edges in any graph is bounded above by $3N-3$. Therefore the length of each edge with respect to the centre of a simplex, is bounded below by $1/(3N-3)$. The same number gives us an obvious lower bound for the translation length of any hyperbolic element.
\end{proof}

\subsection{Connectivity of the Min-set}
In this subsection, we will state some results from \cite{FM18I} and \cite{FM18II}, that we will need in the following section.

\begin{defn}
Let $T,S \in CV_N$. A \textit{simplicial path} between $T,S$ is given by:
\begin{enumerate}
	\item A finite sequence of points $T = X_0, X_1,\ldots, X_k = S$, such that for every $i = 1,\ldots, k$, there is a simplex $\Delta_i$ such that the simplices $\Delta_{X{i-1}}$
	and $\Delta_{X_{i}}$ are both faces of $\Delta_i$.
	
	\item Euclidean segments $\overline{X_{i-1} X_i}$.
\end{enumerate}
The simplicial path is then the concatenation of these Euclidean segments. 
\end{defn}

The following results were proved in \cite{FM18I} and \cite{FM18II}, in the more general context of deformation spaces of free products.
\begin{thm} \label{Connectivity}
	Let $\phi$ be an automorphism of $F_N$.
	\begin{enumerate}
		\item Let $\Delta$ be a simplex of $CV_N$. If $X,Y \in Min(\phi) \cap cl(\Delta)$, where $cl(\Delta) $ is the closure of $\Delta$ in $CV_N$. Then the Euclidean segment $\overline{XY}$ is contained to $Min(\phi)$.
		\item If $\phi$ is irreducible automorphism of $F_N$, then $Min(\phi)$ is connected by simplicial paths in $CV_N$; that is, for every $T,S \in Min(\phi)$ there is a simplicial path $T$ and $S$, which is entirely contained in $Min(\phi)$.
	\end{enumerate}
\end{thm}
\begin{proof}
	
\noindent	
\begin{enumerate}
	\item This is an immediate consequence of the quasi-convexity of the displacement function, see Lemma 6.2 in \cite{FM18I}.
	\item This follows by the Main Theorem of \cite{FM18II} (see Theorem 5.3.) which implies that the set of minimally displaced points for $\phi$ as a subset of the free splitting complex (i.e. the simplicial boardifiaction of $CV_N$) of $F_N$ is connected by simplicial paths, combined with the fact that the Min-Set of an irreducible automorphism does not enter some thin part of $CV_N$ (see the previous subsection). It is easy to see now that, by the connectivity, it's not possible to have minimally displaced points in the boundary without entering any thin part. As a consequence, any minimally displaced point of the free splitting complex must be a point of $CV_N$ and the result follows.
\end{enumerate}
	
\end{proof}

\section{Results}
\subsection{Exponential growth}
Firstly, we will concentrate on the case where $\phi$ is irreducible with exponential growth. In this case, we have the following theorem:

\centralcocompact

\begin{proof}

We will prove that there is a compact set $\mathcal{K}$ of $CV_N$ so that $Min(\phi) \subseteq \mathcal{K} C(\phi)$ and the theorem follows, as $Min(\phi)$ is closed. Therefore as in-going balls are compact (by \ref{Compactness of Balls}), the compact set $\mathcal{K}$ can be chosen to be an in-going closed ball. Let's fix a point $T \in Min(\phi)$. It is then sufficient to prove that there is some positive radius $L$ so that for any $X \in Min(\phi)$, there is some element $\alpha$ of $C(\phi)$ which satisfies $\Lambda(X, (T)\alpha) \leq L$.

We will argue by contradiction. Let's suppose that there is a sequence of points $X_m \in Min(\phi), m=1,2,3,\ldots$, so that $\Lambda(X_m, (T)\alpha) \geq m$, $m=1,2,\ldots$ and for every $\alpha \in C(\phi)$.

Note that there are finitely many $Out(F_N)$- orbits of open simplices on $CV_N$. Therefore, up to taking a subsequence of $X_m$, we can suppose that there is an (open) simplex $\Delta$ and a sequence of (difference of markings) $\psi_m \in Out(F_N)$ so that $X_m \in (\Delta) \psi_m$. 

Firstly, we apply Proposition \ref{Thickness1} for $\phi$ and we get a constant $\epsilon$ (depending on $N$ and on $\phi$), so that $Min(\phi) \subset CV_N(\epsilon)$. Also, by the Remark \ref{Thickness of centre}, the centre of any simplex is $\epsilon$-thick. In particular, the centre $X_{\Delta}$ of $\Delta$ belongs to $CV_N (\epsilon)$ and $(X_m)\psi_m \in \Delta$, so by \ref{Thickness2} there is a constant $M = M(\epsilon) $ (which doesn't depend on $\Delta$ or $m$) which satisfies $\Lambda(X_{\Delta}, (X_m)\psi_m ^{-1} ) \leq M$. It follows that $\Lambda(X_{m},(X_{\Delta})\psi_m  ) \leq M$ for every $m$.

In addition, by assumption, $X_m \in Min(\phi)$ which is equivalent to $\Lambda(X_m,(X_m)\phi) =  \lambda_{\phi} = \lambda$. Therefore, as an easy application of the multiplicative triangle inequality for $\Lambda$ and the fact that $Out(F_N)$ acts on $CV_N$ by isometries with respect to $\Lambda$, the previous relations imply that:
$$
\begin{array}{rcl}
\Lambda((X_{\Delta})\psi_m, (X_{\Delta} ) \psi_m \phi) & \leq &  \Lambda((X_{\Delta})\psi_m, X_m) \Lambda(X_m, (X_m)\phi) \Lambda ((X_m)\phi, (X_{\Delta}) \psi_m \phi )
 \\ \\  & = & \Lambda(X_{\Delta}, (X_m)\psi_m ^{-1}) \Lambda(X_m, (X_m)\phi)  \Lambda (X_m,  (X_{\Delta})\psi_m )  \\ \\ 
  & \leq &  M^2  \lambda.
\end{array}
$$

The previous inequality is equivalent to $\Lambda( X_{\Delta}, (X_{\Delta})\psi_m  \phi \psi_m^{-1} ) \leq M^2  \lambda$ for every $m$. Note that by the Remark \ref{Centres}, $(X_{\Delta})\psi_m  \phi \psi_m^{-1}$ are the centres of the corresponding simplices $(\Delta)\psi_m  \phi \psi_m^{-1}$ for every $m$.

As $CV_N$ is locally finite, there are finitely many simplices so that their centers have bounded distance from $X_{\Delta}$. Therefore, it follows that infinitely many of the simplices $ (\Delta)\psi_n \phi \psi_n^{-1}$,$n=1,2,\ldots$ must be the same, which means that after possibly taking a subsequence of $\psi_n$, we have that $  (X_{\Delta}) \psi_n  \phi \psi_n^{-1}   = (X_{\Delta})\psi_m  \phi \psi_m^{-1}  $ for every $n,m$. As a consequence, the automorphisms  $(\psi_n \phi \psi_n^{-1})(\psi_1  \phi^{-1} \psi_1^{-1})$ fix $X_{\Delta}$  for every $n$.

On the other hand, the stabiliser of any point of $CV_N$ is finite and so infinitely many of these automorphisms are forced to be the same or, in other words, after taking a subsequence, we can suppose that $\psi_m \phi \psi_m^{-1} = \psi_n \phi \psi_n^{-1}$ for every $n,m$. This is equivalent to $\psi_m^{-1} \psi_n \in C(\phi)$ for every $n,m$ and in particular, by fixing one of the indices to be $m=1$, we get that for every $n$ there is some $\alpha_n \in C(\phi)$, so that  $\psi_n =  \psi_1 \alpha_n$.

As a consequence, the $\Lambda$-distance from $(X_{\Delta})\psi_n$ to  $(T)\alpha_n$ does not depend on $n$, as 
$$\Lambda( (X_{\Delta})\psi_n, (T)\alpha_n) = \Lambda( (X_{\Delta}) \psi_1 \alpha_n , (T) \alpha_n) = \Lambda( (X_{\Delta})\psi_1 , T) = C.$$ 

Note that as we have already seen $\Lambda(X_n, (X_{\Delta})\psi_n) \leq M$ for every $n$. Therefore, by applying again the triangle inequality, it follows that $\Lambda( X_n, (T)\alpha_n) \leq M  C$ for every $n$. This contradicts our assumption that $\Lambda( X_n, (T)\alpha_n) \geq n$, as neither $M$ nor $C$ depend on $n$ (note that here by an abuse of the notation, where we still denote by $X_n$ a subsequence $X_{k_n}$ of the original $X_n$, but the inequality still holds because  $k_n \geq n$ for every $n$).
\end{proof}



\begin{rem}
In the previous proof we used irreducibility only in order to ensure the condition that there is some uniform $\epsilon$ so that $Min(\phi) \subset CV_N (\epsilon)$ which follows by the Proposition \ref{Thickness1}. Therefore, we could replace the assumption of irreducibility with this weaker condition.
\end{rem}

For the centralisers of irreducible automorphims with exponential growth, the following theorem holds:


\autococompact

\begin{proof}
	
In light of the previous theorem, it is sufficient to prove that the centraliser of an irreducible automorphism of exponential growth rate is virtually cyclic. 	
	
The case when $\phi$ is irreducible with irreducible powers (iwip) this is well known, by the main result of \cite{BFH-Laminations0}. Note that if $\phi$ is atoroidal and irreducible with $\lambda_{\phi} > 1$, it follows by \cite{Kapovich} that $\phi$ is iwip.

As a consequence, we suppose for the rest of the proof that $\phi$ is toroidal and irreducible. In this case, $\phi$ is a geometric automorphism, i.e. it is induced by a pseudo-Anosov automorphism $f$ of a surface $\Sigma$ with $p\geq1$ punctures,  which acts transitively on the boundary components (this was a folk theorem, until recently where the details appeared in an appendix of \cite{Mut}).
Note that $\phi$ is iwip exactly when $p=1$.

It is well known (see \cite{McCarthy}) that the centraliser $C_{MCG(\Sigma)}(f)$ of the pseudo-Anosov $f$ in $MCG(\Sigma)$ is virtually cyclic. Therefore, it is enough to prove that the centraliser $C(\phi) = C_{Out(F_n)}(\phi)$ of $\phi$ in $Out(F_n)$ is isomorphic to $C_{MCG(\Sigma)} (f)$.

Let's denote by $c_1,\ldots,c_p$ the elements corresponding to the peripheral curves (a simple curve around each puncture). We also denote by $ Out^*(F_n)$, the subgroup of the automorphisms that preserve the set of conjucacy classes of simple peripheral curves (which are $c_i$'s and their inverses). By the Dehn–-Nielsen–-Baer Theorem for surfaces with punctures (see Theorem 8.8. of \cite{FarbMarg}), we have that the natural map from $MCG(\Sigma)$ to $Out^*(F_n)$ is an isomorphism. In other words, an automorphism $\psi$ of $Out(F_n)$ is induced by an element of $MCG(\Sigma)$, exactly when it preserves the set of conjugacy classes of the peripheral curves, or equivalently $\psi \in Out^*(F_n)$.

We will now show that any element of $C(\phi)$ is induced by an element of $C_{MCG(\Sigma)}(f)$ and the proof follows, as any element of $C_{MCG(\Sigma)} (f)$ induces an element of $C(\phi)$.
We can assume that without loss, after re-numbering if needed, that $[\phi(c_i)] = [c_{i+1}]$ mod $p$. If $\psi \in C(\phi)$, then for every $i$, $[\phi \psi (c_i)] = [\psi \phi (c_i) ]= [\psi (c_{i+1})] $ (mod $p$). It follows that $f$ preserves the closed curves $[\psi(c_i)]$ and since $f$ is pseudo-Anosov, we get that any such curve must be a peripheral curve, up to orientation (proper powers can be discounted as $\psi$ is an automorphism). Therefore, $\psi$ must preserve the set of (conjugacy classes) of the peripheral curves and so by the previous criterion, it is induced by an element of $MCG(\Sigma)$.

\end{proof}

As an immediate application of the previous two theorems, we get the following statement which seems to be known to the experts, but it doesn't appear explicitly in the literature:
\begin{thm}
Let $\phi$ be an irreducible automorphism of $F_N$ with $\lambda_{\phi} > 1$. It holds that $Min(\phi) / <\phi>$ is compact.
\end{thm}

\subsection{Growth rate one}
In this subsection, we will cover the case of irreducible automorphisms of growth rate one (it is known that they have finite order). This class has been studied by Dicks and Ventura in \cite{Dicks-Ventura}, where they give an explicit description of any such automorphism.

\begin{nt}
We describe below two types of topological graphs and a graph map for each case:
\begin{itemize}\label{Types of autos}
	\item For the first type, let $p$ be an odd prime. In this case, we define a graph $X_p$ which has rank $p-1$. This graph is consisted of two vertices and $p$ edges which connect them. In order to describe the graph map is more convenient to identify the set of vertices $\{v_0,v_1\}$ with the set  $\{0, 1\}$ and the edge set $\{e_1,...,e_p\}$ with $\mathbb{Z}_p$, while all edges are oriented so that their initial vertex is $v_0$ and their terminal vertex $v_1$. We denote by $\alpha_p$ the graph map, which fixes the vertices and sends $e_i$ to $e_{i+1}$ (mod $p$).

\item For the second type, let $p,q$ two primes  where $p<q$ (here we don't assume that $p$ is odd). We define a graph  $X_{pq}$ of rank $pq-p-q+1$ which is consisted of $p+q$ vertices and $pq$ edges. More specifically, the set of vertices is consisted of two distinct subsets $\{v_1,\ldots,v_p\}$, $\{w_1,\ldots,w_q\}$ which can be naturally identified with $\mathbb{Z}_p \sqcup \mathbb{Z}_q$. Also, for every $i \in \mathbb{Z}_p$ and for every $j \in \mathbb{Z}_q$, there is a unique edge $e_{i,j}$ with initial vertex $v_i$ and terminal vertex $w_j$ and so we can naturally identify the set of edges with $\mathbb{Z}_{p} \times \mathbb{Z}_q$. We denote by $\alpha_{pq}$ the map which sends $e_{i,j}$ to $e_{i+1,j+1}$ (mod$p$, mod$q$, respectively). 
\end{itemize}
\end{nt}
Note that $X_p$, for an odd prime $p$, can be seen as an element of $CV_{p-1}$ by assigning length $1/(p-1)$ on each edge and by considering a marking $R_{p-1} \rightarrow X_p$. Similarly, $X_{pq}$ can be seen as an element of $CV_{pq-p-q+1}$. This is true even when $p=2$, as in this case $X_{2q} = X_q $ (as elements of $CV_{q-1}$ ). We consider this case separately, as the maps $\alpha_q$ and $\alpha_{2q}$ are different ($\alpha_{2q}$ inverts the orientation of each edge, while $\alpha_q$ preserves it).

We can now state the main result of \cite{Dicks-Ventura} (it is proved in Proposition 3.6, even if it is stated in a different form, the following formulation is evident by the proof):
\begin{prop}\label{Classification}
If $\phi$ is an irreducible automorphism of $F_N$ with $\lambda_{\phi} = 1$, then it can be represented by either $\alpha_p$ for some odd prime $p$ or $\alpha_{pq}$ for some primes $p<q$, as in Notation \ref*{Types of autos}.
\end{prop}


We need the following lemma:
\begin{lem}\label{graph auto}
Let $\phi$ be an irreducible automorphism of $CV_N$ with $\lambda_{\phi} = 1$. Let $S$ be a point of $Min(\phi)$ and $f : S \to S$ be an optimal representative of $\phi$. Then $f$ is an isometry on $S$, and hence a graph automorphism.  Morever, $f$ is the unique optimal map representing $\phi$ on $S$. 
\end{lem}

\begin{proof}
Note that as $\lambda_{\phi}=1$,  $Min(\phi)$ is simply the set of fixed points of $\phi$ in $CV_N$. Hence for any loop, $\gamma$, the lengths of $\gamma$ and $f(\gamma)$ are the same (as loops) in $S$, and from there it follows easily that $f$ is an isometry. 	

We can lift this isometry to the universal cover and invoke \cite{CM}, to conclude that this isometry is unique (up to a covering translation), and hence all optimal maps are the same on $Min(\phi)$.

Alternatively, notice that two graph automorphisms give rise to the same action on the associated simplex if and only if the graph maps are the same. 
\end{proof}

\begin{rem*}
	Note that this uniqueness statement is definitely false when $\lambda_{\phi} > 1$; see \cite{FM18I}, Example 3.14.
\end{rem*}


We can now prove our result for the finite order irreducible automorphisms of $F_N$. 


\lambdaonefix 

Note that the existence of such a point $T$ is proved in \cite{Dicks-Ventura}; the content of this Theorem is the uniqueness of such a point. 
\begin{proof}
As note above, the fact that $\lambda_{\phi} = 1$ implies that $Min(\phi) = Fix(\phi)$.

By applying Proposition \ref{Classification}, we get that $\phi$ can be represented by either $\alpha_p$ or $\alpha_{pq}$ (for some primes $p,q$) as an isometry of $X_p$ or $X_{pq}$ (as in notation \ref*{Types of autos}), respectively. For the rest of the proof we will denote the graph map by $\alpha$ and the graph by $T$. In particular, $\Delta_T = \Delta$, then $\phi$ fixes the centre (i.e. if we assign the same length to every edge of $T$) of $\Delta$, as an element of $CV_N$. We will prove that $T$ (with the metric given as above) is the unique fixed point of $CV_N$ for $\phi$.

By the second assertion of Theorem \ref{Connectivity}, $Fix(\phi) = Min(\phi)$ is connected by simplicial paths in $CV_N$. Now consider some other point, $S \in CV_N$ fixed by $\phi$. If we connect $S$ and $T$ by a (simplicial)  path in $Fix(\phi)$, we will be able to produce a point $T' \in Fix(\phi)$ such that either $\Delta'$ is a face of $\Delta$ or vice versa (where $\Delta'$ is the simplex defined by $T'$). 

However, it is clear that $\phi$ cannot fix any other point of $\Delta$; this is because $\alpha$ acts as a cyclic permutation of the edges of $T$, and hence the only metric structure that can be preserved assigns the same length to every edge. Therefore, $\Delta$ must be a face of $\Delta'$. 

We will aim to show that $\Delta$ must be equal to $\Delta'$, which will prove our result. (Since then the only possibilty left is that $S=T$.)

Since $\Delta$ is a face of $\Delta'$, there exists a forest, $F$, whose components are collapsed to produce $T$ from $T'$ (as graphs, absent the metric structure). However, the connectivity of $Fix(\phi)$ allows us to connect $T$ to $T'$ - as metric graphs - via a Euclidean segment in the closure of $\Delta'$ (see the first assertion of \ref{Connectivity}). Thus, without loss of generality, we may assume that the optimal map representing $\phi$ on $T'$ - call this $\alpha'$ - leaves $F$ invariant, since $\alpha'$ is an isometry (by Lemma \ref{graph auto}), and so if the volume of $F$ is sufficiently small, it must be sent to itself.

Our goal is to show that, under the assumption that $T'$ has no valence one or two vertices, each component of $F$ is a vertex. Hence $T=T'$.

Now if we ignore the metric structures we get that on collapsing $F$, $\alpha'$ induces a graph map on the quotient, which is $T$ (as a graph). Since this must also represent $\phi$, we deduce that this induced map is equal to $\alpha$. (Alternatively, collapse $F$ and assign the same length to the surviving edges. The map $\alpha'$ induces an isometry of this graph which is equal to $\alpha$, by \ref{graph auto}.)


Let $G$ denote the cyclic group generated by $\alpha$, acting on $T$. By the comments above, $G$ has an action on $T'$ so that collapsing components of $F$ gives rise to the original action on $T$.

For the remainder of the proof, the edges of $F$ will be called \textit{black}, the edges of the complement of $F$ will be called \textit{white} and vertices that are incident to both and white edges will be called \textit{mixed}. Accordingly, vertices have {\em black} valence and {\em white} valence respectively. 

Now consider a component $C$ of $F$. Note that we can think of $C$ as a {\em vertex} of $T$. Let $v_1, \ldots, v_k$ be the leaves of $C$; that is, those vertices of $C$ with black valence equal to $1$. Since $T'$ can admit no valence one vertices (if we count both black and white edges), each $v_i$ is mixed, incident to both black and white edges. Let $\partial C$ denote the boundary of $C$ in $T'$; that is, the edges (necessarily white) of $T'$ connecting some $v_i$ to a vertex not in $C$. 

Since a vertex stabiliser of $T$ in $G$ acts freely and transitively on the edges incident to it, we deduce that $Stab(C)$ (the set-wise stabiliser) acts freely and transitively on $\partial  C$. Note that the transitive action on $\partial C$ implies that $Stab(C)$ must also act transitively on the leaves, $v_1, \ldots, v_k$.

Now $Stab(C)$ has prime order (since vertex stabilisers in $T$ have prime order), so using the orbit-stabiliser theorem, we deduce that either $k=1$ or $Stab(C)$ acts freely and transitively on $v_1, \ldots, v_k$. In the former case, $C$ consists of a single vertex. In the latter case we get that $k = | \partial C| = |Stab(C)|$. This implies that the white valence of each $v_i$ is equal to $1$. But now the valence of each $v_i$ is exactly $2$, leading to our desired contradiction. 

\medskip

Note that this argument doesn't quite work in the cases where $T$ has a vertex of valence 2 - in the second case of Notation \ref{Types of autos} when $p=2$. Here we get vertices of valence 2 since the graph map $\alpha$ acts as a cyclic permutation on the edges along with an inversion, and we subdivide at the midpoints of edges which are fixed. 

But in this case, the valence 2 vertices are a notational convenience, and we can omit them from $T$, and reach the same conclusion with the same argument.

\end{proof}

The following corollary is now immediate:

\centrallambdaonefinite

\begin{proof}
	It follows immediately by the previous Theorem, by noting that $Min(\phi)$ is $C(\phi)$-invariant.
\end{proof}
In fact, the graph $T$ in the previous corollary, is some graph as in Notation \ref*{Types of autos}, we can get a much more precise description of the centaliser in each case. 
\begin{cor}
If $\phi$ is irreducible automorphism of $F_N$ of growth rate one, then $C(\phi)$ fixes a point $X$, where $X$ is as in Notation \ref{Types of autos}. As a consequence, $C(\phi) = <\phi> \times <\sigma>$, where $\sigma$ is the order two automorphism of $F_N$ that is induced by the graph map of $X$ sending every edge to its inverse.
\end{cor}

\end{document}